\documentclass[10pt, reqno]{amsart}
\usepackage{format}

\def\i{{\rm i}}
\def\Z{\mathbb Z}
\def\N{\mathbb N}

\def\e{{\rm e}}
\def\C{\mathbb C}
\def\X{\mathcal X}
\def\H{\mathcal H}
\def\Y{\mathcal Y}
\def\T{\mathsf T}

\def\G{{\rm G}}

\title{\sffamily Counting degree-constrained orientations}

\author{
  \textbf{Jing Yu}$^{b}$ \and
  \textbf{Jie-Xiang Zhu}$^{a, *}$%
}
\date{} 

\begin{document}

\maketitle

\vspace{-1.2em}

\begin{center}
\footnotesize
${}^{a}$Department of Mathematics, Shanghai Normal University, Shanghai, China\\
${}^{b}$Shanghai Center for Mathematical Sciences, Fudan University, Shanghai, China\\[0.3em]
\end{center}

\footnotetext[1]{%
Corresponding author.\par
E-mail addresses: 
\texttt{jyu@fudan.edu.cn} (J.~Yu),
\texttt{zhujx@shnu.edu.cn} (J.-X.~Zhu).
}

\begin{abstract}
We study the enumeration of graph orientations under local degree constraints. Given a finite graph $G = (V, E)$ and a family of admissible sets $\{\mathsf P_v \subseteq \mathbb{Z} : v \in V\}$, let $\mC(G; \prod_{v \in V} \mathsf P_v)$ denote the number of orientations in which the out-degree of each vertex $v$ lies in $\mathsf P_v$. We prove a general duality formula expressing $\mathcal N(G; \prod_{v \in V} \mathsf P_v)$ as a signed sum over edge subsets, involving products of coefficient sums associated with $\{\mathsf P_v\}_{v \in V}$, from a family of polynomials. Our approach employs gauge transformations, a technique rooted in statistical physics and holographic algorithms. We also present a probabilistic derivation of the same identity, interpreting the orientation-generating polynomial as the expectation of a random polynomial product. As applications, we obtain explicit formulas for the number of even orientations and for mixed Eulerian-even orientations on general graphs. Our formula generalizes a result of Borb\'enyi and Csikv\'ari on Eulerian orientations of graphs.
\end{abstract}

\vskip 0.5cm

 \vskip 0.5cm

\section{Introduction}
Let $G = (V, E)$ be a finite graph. An \emphd{orientation} of $G$ is an assignment of a direction to each edge $\{u, v \} \in E$, turning the original graph into a directed graph. For each $v \in V $, let \( \mathsf{P}_v \subseteq \mathbb{Z} \) be an admissible set associated with \( v \). In this paper, we study the enumeration of orientations of $G$ under the constraint that the out-degree of each vertex $v$ belongs to $\mathsf P_v$, and we denote the total number of such orientations by $\mC (G; \prod_{v \in V} \mathsf P_v)$. By symmetry, reversing the orientation of each edge shows that $\mC (G; \prod_{v \in V} \mathsf P_v)$ also represents the number of orientations of $G$ in which the in-degree of each vertex $v$ lies in $\mathsf P_v$. In what follows, if $\mathsf P_v = \mathsf P \subseteq \Z$ for all $v \in V$, we use the notation $\mC (G; \mathsf P)$ to denote $\mC (G; \prod_{v \in V} \mathsf P_v)$.

Many orientations of interest in the literature are covered by our framework. For instance, an \emphd{Eulerian orientation} of a graph is an orientation in which every vertex has equal in-degree and out-degree. If $d$ is an even number, then the number of Eulerian orientations of a $d$-regular graph $G$  is exactly equal to  $\mC (G; \{ d/2  \})$. The study of Eulerian orientations on regular graphs has a long history: Nagle already derived formulas for 4-regular graphs in 1966 \cite{nagle1966lattice}, and additional related results were later collected in the classical work of Lieb and Wu \cite{LiebWu}. The following result was established by Borb\'enyi and Csikv\'ari \cite{borbenyi2020counting}.

\begin{theo}[{\cite[Theorem~5.1]{borbenyi2020counting}}]\label{Euler orientations, regular}
Let $G = (V, E)$ be a $d$-regular graph. 
Let 
\[F_G(x_0, \ldots, x_d) : = \sum_{A \subseteq E} \left( \prod_{v \in V} x_{d_A(v)}\right),\]
where $d_A(v)$ is the degree of the vertex $v$ in the subgraph $(V, A)$. 
Then $F_G(s_0,\dots ,s_d)$ counts the number of Eulerian orientations of $G$, where 
\[s_k : = \left\{ 
\begin{array}{cc} \frac{\binom{d}{d/2}\binom{d/2}{k/2}}{2^{d/2}\binom{d}{k}} & \mbox{if}\ \ k\ \ \mbox{is even}, \\
0  & \mbox{if}\ \ k\ \ \mbox{is odd}.
\end{array} \right.\]
\end{theo}

For non-regular graphs they also have a corresponding theorem. For details, see \cite[Theorem~5.3]{borbenyi2020counting}. 
Their argument is based on a method known as \emphd{gauge transformations} in statistical physics (see \cite{chertkov2006loop1,chertkov2006loop2}) a.k.a. \emph{holographic reduction} in computer science (see \cite{cai2008basis,cai2017complexity,  cai2008holographic1, cai2008holographic2, cai2009holographic, cai2010symmetric,cai2011holographic,valiant2002expressiveness,valiant2002quantum,valiant2006accidental,valiant2008holographic})  and has been extensively applied in combinatorics in recent years (see e.g., \cite{borbenyi2020counting, bencs2024number}). In the present work, we utilize this method to derive a duality theorem, which is then applied to compute $\mC(G; \prod_{v \in V} \mathsf P_v)$ for some specific choices of $\mathsf P_v \subseteq \Z$.

In this paper, we establish a general duality formula (Theorem \ref{Duality}) that expresses the number of degree-constrained orientations of a finite graph as a signed sum over its edge subsets. Our approach is based on gauge transformations originating from statistical physics and holographic algorithms. To complement this, we provide an independent probabilistic derivation of the same identity by interpreting the orientation-generating polynomial as the expectation of a product of random univariate polynomials. These techniques allow us to derive explicit formulas for counting even orientations, Eulerian-even orientations, and more generally, $N$-divisible orientations. Our results unify and extend earlier work of Borbényi and Csikvári on Eulerian orientations of graphs.

\section{The gauge transformations and duality theorem}
We employ the gauge transformations.
The following notation and terminology are borrowed from \cite{borbenyi2020counting}.
\begin{defn}
A \emphd{normal factor graph} $\H = (V, E, \X, (f_v)_{v\in V})$ is a graph $(V, E)$ equipped with an alphabet $\mathcal X$ and functions $f_v: \X^{d_v} \to \C$ associated with each vertex $v\in V$, where $d_v$ denotes the degree of $v$. The functions $f_v$ are typically referred to as \emphd{local functions}. The set $\X^E$, identified with the collection of all maps from $E$ to $\X$, is called the \emphd{configuration space} of $\H$. The corresponding \emphd{partition function} is defined by
$$
Z(\H) : = \sum_{\sigma \in \X^E} \prod_{v \in V} f_v (\sigma_{\partial v}),
$$
where, for each $\sigma \in \X^E$, $\sigma_{\partial v}$ is the restriction of $\sigma$ to the edges incident to the vertex $v$. Partition functions can indeed be naturally expressed in the Holant problem framework, for example, see \cite{cai2011holographic}.
\end{defn}

Given a normal factor graph $\H = (V, E, \X, (f_v)_{v\in V})$ as defined above. Let $\Y$ be another alphabet, and for each $\{ u, v \} \in E$, let us introduce two matrices $\G_{uv}, \G_{vu} \in \C^{|\Y| \times |\X|}$. We define a family of local functions as follows. For each $v \in V$ with degree $d_v$ and neighborhood $\{u_1, \ldots, u_{d_v} \}$, and $(\tau_{vu_1}, \ldots, \tau_{vu_{d_v}}) \in \Y^{d_v}$, let
$$
\widehat f_v (\tau_{vu_1}, \ldots, \tau_{vu_{d_v}}) : = \sum_{(\sigma_{vu_1}, \ldots, \sigma_{vu_{d_{v}}}) \in \X^{d_v}} \left( \prod_{i = 1}^{d_v} \G_{v u_i} (\tau_{vu_i}, \sigma_{vu_i}) \right) f_v (\sigma_{vu_1}, \ldots, \sigma_{vu_{d_v}}).
$$
Then we can define a new normal factor graph $\widehat \H : = (V, E, \Y, (\widehat f_v)_{v\in V})$. Such a transformation is called a gauge transformation. We will use the following property of this transformation. For its proof, see \cite{chertkov2006loop1,chertkov2006loop2,valiant2008holographic}.

\begin{prop} \label{gauge}
If $\G_{uv}^{\T} \G_{vu} = {\rm Id}_{\X}$ holds for every edge $\{ u, v \} \in E$, then $Z(\H) = Z(\widehat \H)$.
\end{prop}

In what follows, we always take $\X = \Y = \{0, 1\}$. Given $\mathsf P \subseteq \Z$, we define a (linear) functional $\mathcal C(\cdot; \mathsf P)$ on the set of all polynomials as follows: for any polynomial $q(z) \in \Z[z]$,  $\mathcal C(q(z) \,; \,\mathsf P)$ represents the sum of all coefficients of $z^\ell$ in $q(z)$ with $\ell \in \mathsf P$. The following shift property will also be used in the proof. For any $j \in \N$,
\[
\mathcal C (q(z) \, ; \, \mathsf P -j) = \mathcal C(z^j q(z) \, ; \, \mathsf P),
\]
where $\mathsf P - j : = \{ n - j : n \in \mathsf P \}$.

By applying gauge transformations, as described in the proofs of \cite[Lemma 4.3]{borbenyi2020counting} and \cite[Theorem 2.1]{bencs2024number}, to the problem of counting orientations, we obtain the following \emphd{duality theorem}  which transforms the sum over orientations into one over subgraphs.

\begin{theo} \label{Duality}
Let $G = (V, E)$ be a finite graph. For each $v \in V $, let \( \mathsf{P}_v \subseteq \mathbb{Z} \). The total number of orientations of $G$ in which the out-degree of each vertex $v$ belongs to $\mathsf P_v$ is denoted by $\mC (G; \prod_{v \in V} \mathsf P_v)$. Then,
\begin{align*}
\mC (G;  {\textstyle \prod_{v \in V} \mathsf P_v}) = \sum_{F \subseteq E} \frac{(-1)^{|F|}}{2^{|E|}}  \prod_{v \in V} \mathcal C\left((1 - z)^{d_F(v)} (1 + z)^{d_{v} - d_F(v)} \, ; \, \mathsf P_v \right).
\end{align*}
where $d_v$ (respectively $d_F(v)$) is the degree of the vertex $v$ in $G$ (respectively $(V, F)$). 
\end{theo}

\begin{proof}
We begin by constructing a normal factor graph whose partition function encodes the desired count. Let $\mathrm{Sub}(G)$ be the following subdivision of the graph $G$: we put a vertex $e_{u, v}$ to every edge $\{ u, v \}\in E$. The vertex set $V(\mathrm{Sub}(G))$ naturally corresponds to $V \cup E$, and the edge set $E(\mathrm{Sub}(G))$ is $\{\{u, e_{u, v}\} : \{ u, v \} \in E   \}$. The configuration space $\{0 ,1\}^{E(\mathrm{Sub}(G))}$ includes all orientations and subgraphs of $G$. Indeed, for any orientation of $G$, we can define an associated configuration $\sigma : E(\mathrm{Sub}(G)) \to \{0, 1\}$ as follows: for each edge $\{ u, v \}$ with direction $u \to v$, set $\sigma_{ue_{u, v}} = 1$ and $\sigma_{e_{u, v}v} = 0$. For any subgraph of $G$, we can define an associated configuration $\tau : E(\mathrm{Sub}(G)) \to \{0, 1\}$ as follows: for each edge $\{ u, v \} \in E$, if it belongs to the subgraph, set $\tau_{u e_{u, v}} = \tau_{v e_{u, v}} = 1$; otherwise, set $\tau_{ue_{u, v}} = \tau_{ve_{u, v}} = 0$. Thus, the set of all the orientations of $G$ is identified with 
$$\mathcal O(G) : = \big\{ \sigma \in  \{0 ,1\}^{E(\mathrm{Sub}(G))} \, : \, \sigma_{ue_{u, v}} +  \sigma_{v e_{u, v}} = 1 \text{ for each $\{ u, v \} \in E$} \big\};$$
and the set of all the subgraph of $G$ is identified with 
$$\mathcal S (G) : = \big\{ \tau \in  \{0 ,1\}^{E(\mathrm{Sub}(G))} \, : \, \tau_{ue_{u, v}} =  \tau_{v e_{u, v}} \text{ for each $\{ u, v \} \in E$} \big\}.
$$

Now, we define a family of local functions on $\mathrm{Sub}(G)$ as follows: For $v \in V$ with neighborhood $N_G(v) = \{u_1, \ldots, u_{d_v}  \}$, 
\[f_v ( \sigma_{ve_{v,u_1}},\dots ,\sigma_{ve_{v,u_{d_v}}}  ) :=
\begin{cases}
1 & \mbox{if}\ \ \sum_{u \in N_G(v)}\sigma_{ve_{v,u}} \in \mathsf P_v, \\
0 & \mbox{if}\ \ \sum_{u \in N_G(v)}\sigma_{ve_{v,u}} \notin \mathsf P_v;
\end{cases}
\]
and for $e_{u, v} \in E$,
\[f_{e_{u,v}} (\sigma_{ue_{u,v}},\sigma_{ve_{u,v}}) :=
\begin{cases}
1 & \mbox{if}\ \ \sigma_{ue_{u,v}} + \sigma_{ve_{u,v}} =1, \\
0 & \mbox{otherwise.}
\end{cases}
\]
Since for any $\sigma \in \mathcal O(G)$, $\sum_{u \in N_G(v)}\sigma_{ve_{v,u}}$ represents the out-degree of $\sigma$ at vertex $v$. By definition, $\mC (G; \prod_{v \in V} \mathsf P_v)$ is the partition function of the normal factor graph $\big( \mathrm{Sub}(G), \{0, 1\}, (f_{v})_{v \in V(\mathrm{Sub}(G))} \big)$.

Next we apply the gauge transformation. For each edge $e= \{ u,v\} \in E(G)$ we introduce two matrices in $\mathrm{Sub}(G)$:
$\G_{eu} = \G_{ev} = \G_1$ and $\G_{ue} = \G_{ve} = \G_2$, where
\[\G_1:=\frac{1}{\sqrt{2}}\left( \begin{array}{cc} 1 & 1 \\ \i& -\i \end{array}\right)\ \ \ \mbox{and} \ \ \ \G_2:=\frac{1}{\sqrt{2}}\left( \begin{array}{cc} 1 & 1 \\ -\i & \i \end{array}\right) \ \ \ \mbox{with} \ \ \i = \sqrt{-1}.\]
The rows and columns of $\G_1,\G_2$ are indexed by $0$ and $1$. Taking the gauge transformation, let us begin with $\widehat f_{e}$, where $e \in E$. A simple computation gives
\begin{align*}\widehat f_{e}(\tau_1,\tau_2) = \sum_{\sigma_1,\sigma_2 \in\{0, 1\}}\G_1(\tau_1,\sigma_1)\G_1(\tau_2,\sigma_2) f_{e}(\sigma_1,\sigma_2) = \begin{cases}
1 & \mbox{if}\ \ \tau_1 = \tau_2, \\
0 & \mbox{otherwise.}
\end{cases}
\end{align*}
This implies that, in the formulation of the new partition function, the summation is restricted to configurations corresponding to subgraphs of $G$. Then we turn to compute $\widehat f_{v} (\tau_1,\dots ,\tau_{d_v})$. By definition
\begin{align} \label{local}
\widehat f_{v}(\tau_1,\dots ,\tau_{d_v})=\sum_{\sigma_1,\dots ,\sigma_{d_v} \in \{0, 1\} }\prod_{i=1}^{d_v} \G_2(\tau_i,\sigma_i)f_{v}(\sigma_1,\dots ,\sigma_{d_v}).
\end{align}
Set $m : =  \sum_{i = 1}^{d_v} \sigma_i$ and $k: = \sum_{i = 1}^{d_v} \tau_i$. Recall that only those terms with $m \in \mathsf P_v$ remain. If $k = 0$, then
\[
\widehat f_{v}(\tau_1,\dots ,\tau_{d_v}) = \frac{1}{2^{{d_v}/2}}\sum_{\substack{0 \leq m \leq d_v\\ m \in \mathsf P_v}} \binom{d_v}{m} = \frac{\mathcal C \big((1 + z)^{d_v}; \mathsf P_v \big)}{2^{{d_v}/2}}.
\]
Now assume that $k\in\{1,\ldots, d_v\}$, where $k = \sum_{i = 1}^{d_v} \tau_i$. Notice that if there are $j$ positions where both $\sigma_i=\tau_i=1$, then its contribution to the sum \eqref{local} is $\i^j(-\i)^{k-j} = (-1)^{k-j} \i^k$. Hence,
\[
\widehat f_{v} (\tau_1,\dots ,\tau_{d_v}) = \sum_{\sigma_1,\dots ,\sigma_{d_v} \in\{0, 1\}}\prod_{i=1}^{d_v} \G_2(\tau_i,\sigma_i)f_v(\sigma_1,\dots ,\sigma_{d_v}) =\frac{1}{2^{{d_v}/2}}\sum_{\substack{0 \leq m \leq d_v\\ m \in \mathsf P_v}}\sum_{j=0}^{k} \binom{k}{j} \binom{d_v-k}{m-j}(-1)^{k-j}\i^k,
\]
where by convention, $\binom{x}{y} = 0$ if $y < 0$ or $y > x$. Noting that when $0 \leq m \leq d_v$, $m - j$ runs over all nonnegative integers between $0$ and $d_v - k$, by exchanging the order of summation we deduce 
\begin{align*}
\widehat f_{v} (\tau_1,\dots ,\tau_{d_v}) = \frac{(-\i)^k}{2^{{d_v}/2}}  \sum_{j=0}^{k} (-1)^j \binom{k}{j} \cdot \mathcal C\big( (1 + z)^{d_v - k} \, ; \, \mathsf P_v-j  \big) = \frac{(-\i)^k}{2^{{d_v}/2}}  \sum_{j=0}^{k} (-1)^j \binom{k}{j} \cdot \mathcal C\big( z^j (1 + z)^{d_v - k}  \, ; \, \mathsf P_v  \big).
\end{align*}
Due to the fact that $\mathcal C(\cdot, \mathsf P_v)$ is linear, we further obtain
\begin{align*}
\widehat f_{v} (\tau_1,\dots ,\tau_{d_v})  = \frac{(-\i)^k}{2^{d_v/2}} \cdot \mathcal C\big( (1 - z)^k (1 +z)^{d_v -k} \, ; \,  \mathsf P_v   \big).
\end{align*}
Since $\G_2^{\T} \G_1 = {\rm Id}$, it follows from Proposition \ref{gauge} that,
\begin{align*}
\mC (G;  {\textstyle \prod_{v \in V} \mathsf P_v}) 
&= \sum_{F \subseteq E} \frac{(-\i)^{\sum_{v\in V} d_F(v)}}{2^{\sum_{v\in V} d_v /2}}  \prod_{v \in V} \mathcal C\big( (1 - z)^{d_F(v)} (1 +z)^{d_v - d_F(v)} \, ; \,  \mathsf P_v   \big)\\
&= \sum_{F \subseteq E} \frac{(-1)^{|F|}}{2^{|E|}}  \prod_{v \in V} \mathcal C\big( (1 - z)^{d_F(v)} (1 +z)^{d_v - d_F(v)} \, ; \,  \mathsf P_v   \big),
\end{align*}
where we used the handshaking lemma. The proof is complete.
\end{proof}

\section{A probabilistic point of view}
We extend the definition of $\mathcal C(\cdot; \mathsf P)$ to the multivariate setting. Given $\mathsf S \subseteq \Z^k$, for any polynomial $q(z_1, \ldots, z_k)$, $\mathcal C\big(q(z_1, \ldots, z_k) \, ; \, \mathsf S  \big)$ represents the sum of all coefficients of $z_1^{\ell_1}\cdots z_k^{\ell_k}$ in $q(z_1, \ldots, z_k)$ with $(\ell_1, \ldots, \ell_k) \in \mathsf S$. Then $\mathcal C$ satisfies the following factorization property. For any $\mathsf S_1, \ldots, \mathsf S_k \subseteq \Z$,
\begin{align} \label{product}
\mathcal C \left( \prod_{j = 1}^k q_j(z_j) \, ; \, \prod_{j = 1}^k \mathsf S_j     \right) = \prod_{j = 1}^{k} \mathcal C( q_j(z_j) \, ; \, \mathsf S_j).
\end{align}

Let $G = (V, E)$ be a finite graph. We associate to each vertex $v \in V$ a variable $z_v$, and consider the following graph polynomial
\[
H_G({\bf z}) : = \prod_{\{ u,v \} \in E}(z_u + z_v).
\]
Then we have the following proposition, see also  \cite[page 3]{csikvari2022short}.

\begin{prop} \label{polynomial}
With the notation introduced above. For each $v \in V$, let $\mathsf P_v \subseteq \Z$. Then,
\[
\mC (G;  {\textstyle \prod_{v \in V} \mathsf P_v}) = \mathcal C (H_G({\bf z});  {\textstyle \prod_{v \in V} \mathsf P_v})
\]
\end{prop}

\begin{proof}
Indeed, for any orientation of $G$, identify a corresponding term in the expansion of $H_G$ as follows: for each edge $\{ u, v\}$, if it is oriented as $u \to v$, select $z_u$ from $(z_u + z_v)$; if it is oriented as $v \to u$, select $z_v$. In this way, each orientation of $G$ corresponds bijectively to a term in the expansion of $H_G$, with the exponent vector encoding the out-degree of the corresponding orientation at each vertex. The proposition then follows immediately.
\end{proof}

Now we consider the sum
\[\sum_{F \subseteq E} \frac{(-1)^{|F|}}{2^{|E|}}  \prod_{v \in V} \mathcal C\left((1 - z)^{d_F(v)} (1 + z)^{d_{v} - d_F(v)} \, ; \, \mathsf P_v \right),
\]
which, by the factorization property \eqref{product}, equals to
\[
 \sum_{F \subseteq E} \frac{1}{2^{|E|}}   \mathcal C\bigg((-1)^{|F|}\prod_{v \in V}(1 - z_v)^{d_F(v)} (1 + z_v)^{d_{v} - d_F(v)} \, ; \, \mathsf \prod_{v \in V} \mathsf P_v \bigg).
\]
For each edge $e \in E$, we define a function  $\chi_e$ that maps subsets of $E$ to $\{ \pm 1 \}$ as follows: 
\[\chi_e(F) :=
\begin{cases}
-1 & \mbox{if}\ \ e \in F, \\
1 & \mbox{if}\ \ e \notin F.
\end{cases}
\]
Then for every $F \subseteq E$ and $v \in V$, we have
\[
(1 - z_v)^{d_F(v)} (1 + z_v)^{d_{v} - d_F(v)} = \prod_{\{ e \in E \, : \, v \in  e  \}} (1 + \chi_e(F) z_v).
\]
Taking into account that $(-1)^{|F|} = \prod_{e \in E} \chi_e(F)$, we yield
\begin{equation} \label{e0}
\begin{split}
& \sum_{F \subseteq E} \frac{1}{2^{|E|}}   \mathcal C\bigg((-1)^{|F|}\prod_{v \in V}(1 - z_v)^{d_F(v)} (1 + z_v)^{d_{v} - d_F(v)} \, ; \, \mathsf \prod_{v \in V} \mathsf P_v \bigg)\\
 & =  \sum_{F \subseteq E} \frac{1}{2^{|E|}} \mathcal C\bigg(\prod_{e \in E} \chi_e(F) \prod_{v \in V}\prod_{\{ e \in E \, : \, v \in  e  \}} \big( 1 + \chi_e(F) z_v \big) \, ; \, \mathsf \prod_{v \in V} \mathsf P_v \bigg)\\
& =  \mathcal C\bigg(\sum_{F \subseteq E} \frac{1}{2^{|E|}} \prod_{e = \{u, v\} \in E} \chi_{e}(F)  \big( 1 + \chi_{e}(F) z_u \big) \big( 1 + \chi_{e}(F) z_v \big) \, ; \, \mathsf \prod_{v \in V} \mathsf P_v \bigg),
\end{split}
\end{equation}
where in the last step we use the elementary identity
\[
\prod_{v \in V}\prod_{\{ e \in E \, : \, v \in  e  \}} \big( 1 + \chi_e(F) z_v \big) =  \prod_{e = \{u, v\} \in E} \big( 1 + \chi_{e}(F) z_u \big) \big( 1 + \chi_{e}(F) z_v \big),
\]
and the linearity of $\mathcal C(\cdot, \prod_{v \in V} \mathsf P_v)$.

Notice that the sums in \eqref{e0} are taken over all possible sign choices of $\{ \chi_e :  e \in E  \}$. This observation motivates the introduction of a sequence of independent Bernoulli random variables $\varepsilon_e$ indexed by $e \in E$; that is,  $\mathbb P(\varepsilon_e = 1) = \mathbb P(\varepsilon_e = -1) = \frac12$.
Then we have
\begin{align} \label{f1}
\sum_{F \subseteq E} \frac{1}{2^{|E|}} \prod_{e = \{u, v \} \in E} \chi_{e}(F)  \big( 1 + \chi_{e}(F) z_u \big) \big( 1 + \chi_{e}(F) z_v \big) = \mathbb E \left[ \prod_{e = \{u, v \} \in E} \varepsilon_e  (1 + \varepsilon_e z_u) (1 + \varepsilon_e z_v)\right].
\end{align}
Since $\varepsilon_e$'s are i.i.d. Bernoulli random variables, 
\begin{equation}
\begin{split}\label{f2}
\mathbb E \left[ \prod_{e = \{ u, v\} \in E} \varepsilon_e  (1 + \varepsilon_e z_u) (1 + \varepsilon_e z_v)\right] &=  \prod_{ e = \{ u, v\} \in E} \mathbb E \big[ \varepsilon_e  (1 + \varepsilon_e z_u) (1 + \varepsilon_e z_v) \big] \\
& = \prod_{\{ u, v \} \in E} (z_u + z_v) = H_G({\bf z}).
\end{split}
\end{equation}

Combining \eqref{e0}-\eqref{f2} and using Proposition \ref{polynomial}, we thus obtain an alternative proof of Theorem \ref{Duality}. The key to this argument is identity \eqref{f2}: by introducing randomness, we express $H_G$ as the expectation of a family of random polynomials, each of which can be factorized into a product of univariate polynomials in $z_v$. Clearly, the identity \eqref{f2} has the following generalization. Let $\{ \epsilon_e\}_{e \in E}$ and $\{ \tilde{\epsilon}_e \}_{e \in E}$ be two sequences of independent random variables defined on some probability space $(\Omega, \mathbb P)$, and assume that for each $e \in E$,
\[
\mathbb E [\tilde\epsilon_e] = \mathbb E [\tilde\epsilon_e \epsilon_e^2] = 0, \quad \mathbb E [\tilde\epsilon_e \epsilon_e] = 1.
\]
Then,
\begin{align}
\mathbb E \left[ \prod_{e = \{u, v\} \in E} \tilde\epsilon_e  (1 + \epsilon_e z_u) (1 + \epsilon_e z_v)\right] = H_G({\bf z}).
\end{align}

In applications, we mostly consider the case where $\epsilon_e$ and $\tilde \epsilon_e$ take finitely many values. We thus have the following generalization of Theorem \ref{Duality}.

\begin{theo} \label{Duality'}
With the same notation as in Theorem \ref{Duality}. Let $N \in \N$ with $N \ge 2$, and let $\{\alpha(j)\}_{j = 1}^N$, $\{\beta(j)\}_{j = 1}^N$ be two sequences of complex numbers, satisfying
\[
\sum_{j=1}^{N}  \alpha(j) = \sum_{j=1}^{N}  \alpha(j) \beta(j)^2 = 0, \quad \sum_{j = 1}^{N}  \alpha(j) \beta(j) = 1.
\]
Then,
\[
\mC (G;  {\textstyle \prod_{v \in V} \mathsf P_v}) = \sum_{f \in [N]^E} \, \prod_{e\in E}  \alpha(f(e)) \prod_{v \in V} \mathcal C \bigg( \prod_{\{e\in E \, : \, v\in e  \}} \big(1 + \beta(f(e)) z \big) \, ; \,  \mathsf P_v  \bigg),
\]
where $[N]^E$ denotes the set of all $N$-colorings of the edge set $E$, that is, all maps from $E$ to $\{ 1, \ldots, N\}$.
\end{theo}

In fact, the random variables $\epsilon_e$ and $\tilde \epsilon_e$ can be chosen with considerable flexibility: independence is the only essential requirement, and their distributions need not be identical. This flexibility allows for edge-dependent distributions, which might be useful in specific applications. Furthermore, continuous-valued choices (e.g., with density functions) are also permissible, which corresponds to maps from $E$ to $\mathbb R$. Such constructions seem difficult to achieve within the framework of gauge transformations.

\section{Applications in graph theory}

\subsection{$N$-divisible orientations}

\begin{defn}
An orientation of a graph is called a \emphd{$N$-divisible orientation} if the out-degree is divisible by $N$ at each vertex. In particular, we refer to a $2$-divisible orientation as an \emphd{even orientation}.
\end{defn}

In this subsection, we consider the counting problem of $N$-divisible orientations. With the notation above, we need to compute $\mC (G; N\Z)$. Before proceeding, we introduce the following property of a function $f : E \to \N$ with an integer $N \ge 2$:
$$
\max_{v \in V} \big| \{ f(e) \, : \, v \in e   \} \big| \le N - 1 \eqno{(H_N)}
$$
And we define
\[
\mathcal A_N := \left\{ f \in [N]^E \, : \, f \ \  \text{satisfies property $(H_N)$ } \right\}.
\]

For $j = 1, \ldots, N$, we define
\[
\omega_N (j) : = \e^{\frac{2j\pi \i}{N}}.
\]
Then it is easy to verify that
\begin{equation} \label{sum0}
\mathcal C(q(z) \, ; \, N \Z) = \frac1N \sum_{j=1}^{N} q(\omega_N(j)).
\end{equation}
To apply Theorem \ref{Duality'}, we define  $\{\alpha(j)\}_{j = 1}^N$ and $\{\beta(j)\}_{j = 1}^N$ as follows: 
\[
\alpha(j) = -\frac{\e^{\frac{2j\pi \i}{N}}}{N} = -\frac{\omega_N(j)}{N}, \,  \quad \beta(j) = - \e^{\frac{-2j\pi \i}{N}} = -\overline{\omega_N(j)}.
\]
Then $\{\alpha(j)\}_{j = 1}^N$ and $\{\beta(j)\}_{j = 1}^N$ satisfy the conditions in Theorem \ref{Duality'}. By Theorem \ref{Duality'} and \eqref{sum0}, we obtain
\begin{equation} \label{sum}
\begin{split}
\mC (G;  N\Z) &= \sum_{f \in [N]^E} \, \prod_{e\in E}  \alpha(f(e)) \prod_{v \in V} \mathcal C \bigg( \prod_{\{e\in E \, : \, v\in e  \}} \big(1 + \beta(f(e)) z \big) \, ; \,  N\Z  \bigg)\\
& = \frac{(-1)^{|E|}}{N^{|E|+|V|}} \sum_{f \in [N]^E} \, \prod_{e\in E}  \omega_N(f(e)) \prod_{v \in V}   \bigg\{ \sum_{j = 1}^N \prod_{\{e\in E \, : \, v\in e  \}} \big(1 - \overline{\omega_N(f(e))} \omega_N(j) \big)  \bigg\}.
\end{split}
\end{equation}
Note that the roots of the polynomial 
$$\prod_{\{e\in E \, : \, v\in e  \}} \big(1 - \overline{\omega_N(f(e))}  z \big)$$ 
are exactly the set $\{ \omega_N(f(e)) \, : \, v \in e \}$. Consequently, if for some $f \in [N]^E$ and $v \in V$,
\[
\big| \{ f(e) \, : \, v \in e   \} \big| = N,
\]
then we have
\[
 \sum_{j = 1}^N \prod_{\{e\in E \, : \, v\in e  \}} \big(1 - \overline{\omega_N(f(e))} \omega_N(j) \big) = 0.
\]
Therefore, the sum on the right-hand side of \eqref{sum} is restricted to all $f \in \mathcal A_N$. Finally, we yield
\begin{align}
\mC (G;  N\Z) = \frac{(-1)^{|E|}}{N^{|E|+|V|}} \sum_{f \in \mathcal A_N} \, \prod_{e\in E}  \omega_N(f(e)) \prod_{v \in V}   \bigg\{ \sum_{j \in [N] \backslash \{ f(e) :  v \in e   \}} \prod_{\{e\in E  :  v\in e  \}} \big(1 - \overline{\omega_N(f(e))} \omega_N(j) \big)  \bigg\}.
\end{align}
This formula reveals a relationship between the number of $N$-divisible orientations and the $N$-colorings of the edge set. However, it seems difficult to simplify further and is not effective for practical computations. It is worth noting that, in general, property $(H_N)$ does not imply that $f(E)$ uses at most $N - 1$ colors. However, in the special case $N = 2$, property $(H_2)$ together with the connectivity of the graph, implies that all edges are assigned the same color. This observation allows for a further simplification of the formula in this case. We thus have the following result.

\begin{corl} \label{EVEN}
Let $G =(V, E)$ be a finite connected graph. Then the total number of even orientations of $G$ is equal to
$$
2^{|E| - |V|} \big( 1 + (-1)^{|E|} \big).
$$
\end{corl}

\begin{proof}
With the notation above, we need to compute $\mC (G; 2\Z)$. Noting that 
\[
\mathcal C(q(z) \, ; \, 2\Z) = \frac{q(1) + q(-1)}{2},
\]
we immediately yield that for $d \geq 1$ and $0 \leq k \leq d$,
\begin{equation} \label{e1}
\begin{split}
\mathcal C \big( (1 - z)^k (1 + z)^{d - k} \, ; \, 2\Z \big) = \begin{cases}
2^{d - 1} & \mbox{if}\ \ k = 0 \, \, \, {\rm or }\, \, \, d; \\
0 & \mbox{otherwise}.
\end{cases}
\end{split}
\end{equation}
Combining this with Theorem \ref{Duality}, it follows that in the summation expressions for $\mC (G; 2\Z)$, only the subgraphs $F$ satisfying $d_F(v) = 0$ or $d_v$ for all $v \in V$ are involved. Moreover, based on the connectedness of $G$, we can conclude that the subgraphs $F$ that satisfy these conditions must be either $\emptyset$ or $E$. Consequently,
\begin{align*}
\mC(G; 2\Z) = \frac{1}{2^{|E|}} 2^{ \sum_{v\in V} d_v - |V|} + \frac{(-1)^{|E|}}{2^{|E|}} 2^{\sum_{v\in V} d_v - |V|},
\end{align*}
The proof is then complete after applying the handshaking lemma.
\end{proof}

\begin{remk}
 To the best of our knowledge, Corollary~\ref{EVEN} is not listed in any paper, however it can be viewed as a corollary of Exercise 5.16 in \cite{lovasz2007combinatorial} by noticing that there is a bijection between the even spanning subgraphs and even orientations of the given connected graph. 
\end{remk}

\subsection{Mixed Eulerian-even orientations}
\begin{defn}
For a finite graph $G=(V, E)$ with $V = V_1 \bigsqcup V_2$, where every vertex in $V_1$ has an even degree, an orientation of $G$ is called a \emphd{mixed Eulerian-even orientation w.r.t. $(V_1, V_2)$} if every vertex $v \in V_1$ has equal in-degree and out-degree, and every vertex in $V_2$ has even out-degree.
\end{defn}
According to Corollary~\ref{EVEN}, for a connected graph, even orientations exist if and only if $|E|$ is even---that is, their existence depends on the global quantity $|E|$. In contrast, Eulerian orientations require a finer condition: each vertex must have even degree. In this section, we study the enumeration of orientations that mix these two types and provide a sufficient condition for their existence.

For an even integer $d$ and $0 \leq k \leq d$, we have
\begin{align*}
\mathcal C \big( (1- z)^k (1 +z)^{d - k} \, ; \, \{d/2\}  \big) = \sum_{j_1 + j_2 = d/2} (-1)^{j_1} \binom{k}{j_1} \binom{d - k}{j_2} = \frac{\binom{d}{d/2}}{\binom{d}{k}} \sum_{j_1=0}^{k} (-1)^{j_1} \binom{d/2}{j_1} \binom{d/2}{k - j_1}.
\end{align*}
Since $\sum_{j_1=0}^{k} (-1)^{j_1} \binom{d/2}{j_1} \binom{d/2}{k - j_1}$ is the coefficient of $z^{k}$ in $(1 - z)^{d/2}(1 + z)^{d/2} = (1 - z^2)^{d/2}$, we obtain
\begin{equation} \label{e2}
\begin{split}
\mathcal C \big( (1- z)^{k} (1 +z)^{d - k} \, ; \, \{d/2\}  \big) = \left\{ 
\begin{array}{cc} (-1)^{k/2}\frac{\binom{d}{d/2}\binom{d/2}{k/2}}{\binom{d}{k}} & \mbox{if}\ \ k\ \ \mbox{is even}; \\
0  & \mbox{if}\ \ k\ \ \mbox{is odd}.
\end{array} \right.
\end{split}
\end{equation}
Applying Theorem \ref{Duality} and using \eqref{e1}, \eqref{e2}, we know that
\begin{equation} \label{e3}
\begin{split}
& \mC \big(  G \, ; \, {\textstyle \prod_{v \in V_1}} \{ d_v/2 \} \times {\textstyle \prod_{v \in V_2}} 2\Z \big) = \sum_{F \subseteq E} \frac{(-1)^{|F|}}{2^{|E|}} \prod_{v \in V_1} (-1)^{d_F(v)/2}\frac{\binom{d_v}{d_v/2}\binom{d_v/2}{d_F(v)/2}}{\binom{d_v}{d_F(v)}} \cdot \prod_{v \in V_2} 2^{d_v - 1}\\
& = \left(\prod_{v \in V_1} \frac{\binom{d_v}{d_v/2}}{2^{d_v/2}}\right) \cdot 2^{\frac12 \sum_{v \in V_2} d_v - |V_2|} \cdot \bigg( \sum_{F \subseteq E} \prod_{v \in V_1}\frac{\binom{d_v/2}{d_F(v)/2}}{\binom{d_v}{d_F(v)}} \prod_{v\in V_2} (-1)^{d_F(v)/2} \bigg),
\end{split}
\end{equation}
where the sum is taken over all $F \subseteq E$ satisfying the following two conditions:
\begin{enumerate}
\item[(i)] for any $v \in V_1$, $d_F(v)$ is even;

\item[(ii)] for any $v \in V_2$, $d_F(v)$ is 0 or $d_v$;
\end{enumerate}
we denote the collection of all such $F$ by $\mathcal B(V_1, V_2)$. Now we have the following theorem.

\begin{theo} \label{mix}
Let $G = (V, E)$ be a finite graph. For the decomposition $V = V_1\bigsqcup V_2$, where $V_1$ and $V_2$ satisfy 
\begin{enumerate}
\item[$\bullet$] every vertex $v \in V_1$ has an even degree;

\item[$\bullet$] after decomposing the induced subgraph $G[V_2]$ into its components $G[V_2^{(1)}], \ldots, G[V_2^{(m)}]$, each $V_2^{(i)}$ satisfies $4 \mid \sum_{v\in V_2^{(i)}} d_v$.

\end{enumerate}
Then there exists a mixed Eulerian-even orientation w.r.t. $(V_1, V_2)$. Moreover, the total number of such orientations is equal to
\[
 \left( \prod_{v \in V_1} \frac{\binom{d_v}{d_v/2}}{2^{d_v/2}} \right) \cdot 2^{\frac12 \sum_{v \in V_2} d_v - |V_2|} \cdot \sum_{F \in \mathcal B(V_1, V_2)} \prod_{v \in V_1}\frac{\binom{d_v/2}{d_F(v)/2}}{\binom{d_v}{d_F(v)}},
\]
which is bounded from below by
\[
\left( \prod_{v \in V_1} \frac{\binom{d_v}{d_v/2}}{2^{d_v/2}} \right) \cdot 2^{\frac12 \sum_{v \in V_2} d_v - |V_2|}.
\]
\end{theo}

\begin{proof}
For each $i = 1, \ldots, m$, by the connectedness of $(V_2^{(i)}, E)$, the subgraph $F$ that satisfies condition (ii) must have either $d_F(v) = 0$ for all $v \in V_2^{(i)}$, or $d_F(v) = d_v$ for all $v \in V_2^{(i)}$. Regardless of which case holds, since $4 \mid \sum_{v \in V_2^{(i)}} d_v$, we always have 
\[
\prod_{v\in V_2} (-1)^{d_F(v)/2} = \prod_{i = 1}^{m} \prod_{v\in V_2^{(i)}} (-1)^{d_F(v)/2} = 1.
\]
Consequently, it follows from \eqref{e3} that
\begin{align*}
\mC \big(  G \, ; \, {\textstyle \prod_{v \in V_1}} \{ d_v/2 \} \times {\textstyle \prod_{v \in V_2}} 2\Z \big) &= \left(\prod_{v \in V_1} \frac{\binom{d_v}{d_v/2}}{2^{d_v/2}} \right) \cdot 2^{\frac12 \sum_{v \in V_2} d_v - |V_2|} \cdot \sum_{F \in \mathcal B(V_1, V_2)} \prod_{v \in V_1}\frac{\binom{d_v/2}{d_F(v)/2}}{\binom{d_v}{d_F(v)}}\\
& \geq \left( \prod_{v \in V_1} \frac{\binom{d_v}{d_v/2}}{2^{d_v/2}} \right) \cdot 2^{\frac12 \sum_{v \in V_2} d_v - |V_2|},
\end{align*}
which completes the proof.
\end{proof}

\begin{remk}
It should be pointed out that the second condition in Theorem \ref{mix} is only a sufficient one for the existence of a mixed Eulerian-even orientation, but it is not necessary. To see this, let us consider the graph
$$
G=(V,E),\qquad  
V=\{v_1,v_2,v_3,v_4\},\quad  
E=\big\{\{v_1,v_2\},\ \{v_1,v_4\},\ \{v_2,v_3\}\big\},
$$
and let
$$
V_1=\{v_1\},\qquad V_2=\{v_2,v_3,v_4\}.
$$
Since $d_{v_1} = 2$ is even, the first condition of Theorem~\ref{mix} is satisfied. However, the induced subgraph $G[V_2]$ consists of two components, namely $\{v_2,v_3\}$ and $\{v_4\}$, whose degree sums in $G$ are $3$ and $1$, respectively.  Neither of them is divisible by $4$, so the second condition does not hold.
Nevertheless, orienting the edges as
$$
v_2 \to v_1,\qquad v_1 \to v_4,\qquad v_2 \to v_3
$$
yields a mixed Eulerian-even orientation w.r.t. $(V_1,V_2)$.
\end{remk}

\section*{Acknowledgements}
J.Y. acknowledges partial support from the China Postdoctoral Science Foundation grant GZC20252041 and the National Natural Science Foundation of China grant 12371343 (PI: Hehui Wu).   
J.-X.Z. acknowledges
support from the Natural Science Foundation of
Shanghai (25ZR1402414) and the NSFC (12501185). 
We also thank the anonymous referees for reading the manuscript carefully and providing helpful comments.

\printbibliography

@article{nagle1966lattice,
  title={Lattice statistics of hydrogen bonded crystals. I. The residual entropy of ice},
author={Nagle, J. F.},
  journal={J. Math. Phys.},
  volume={7},
  number={8},
  pages={1484--1491},
  year={1966},
  publisher={American Institute of Physics}
}

@article{LiebWu,
	author = {E. H. Lieb and F. Y. Wu},
	title = {Two-Dimensional Ferroelectric Models},
	journaltitle = {Phase Transitions and Critical Phenomena},
    volume={1},
	editor = {Domb, C. and Green, M.},
	pages = {331--490},
	date = {1972},
publisher={Academic Press}
}

@unpublished{bencs2024number,
  title={Number of Eulerian orientations for Benjamini--Schramm convergent graph sequences},
  author={Bencs, F. and Borb{\'e}nyi, M. and Csikv{\'a}ri, P.},
  date={2024},
  howpublished = {\url{https://arxiv.org/abs/2409.18012} (preprint)},
}

@article{borbenyi2020counting,
  title={Counting degree-constrained subgraphs and orientations},
  author={Borb{\'e}nyi, M. and Csikv{\'a}ri, P.},
  journal={Discrete Math.},
  volume={343},
  number={6},
  pages={111842},
  year={2020},
  publisher={Elsevier}
}

@article{valiant2002expressiveness,
  title={Expressiveness of matchgates},
  author={Valiant, L.G. },
  journal={Theoret. Comput. Sci.},
  volume={289},
  number={1},
  pages={457--471},
  year={2002},
  publisher={Elsevier}
}

@article{valiant2002quantum,
  title={Quantum computers that can be simulated classically in polynomial time},
  author={Valiant, L.G. },
  journal={SIAM J. Comput.},
  volume={31},
  number={4},
  pages={1229--1254},
  year={2002},
  publisher={Society for Industrial and Applied Mathematics Philadelphia, PA, USA}
}

@article{valiant2006accidental,
  title={Accidental algorithms},
  author={Valiant, L.G. },
  journal={47th Annual IEEE Symposium on Foundations of Computer Science (FOCS)},
  year={2006},
  pages={509--517},
}

@article{valiant2008holographic,
  title={Holographic algorithms},
  author={Valiant, L.G. },
  journal={SIAM J. Comput.},
  volume={37},
  number={5},
  pages={1565--1594},
  year={2008},
  publisher={SIAM}
}

@article{cai2008holographic1,
  title={Holographic algorithms: guest column},
  author={Cai, J.-Y.},
  journal={ACM SIGACT News},
  volume={39},
  number={2},
  pages={51--81},
  year={2008},
  publisher={ACM New York, NY, USA}
}

@book{cai2017complexity,
  title={Complexity Dichotomies for Counting Problems: Volume 1, Boolean Domain},
  author={Cai, J.-Y. and Chen, X.},
  year={2017},
  publisher={Cambridge University Press}
}

@article{cai2009holographic,
  title={Holographic algorithms: The power of dimensionality resolved},
  author={Cai, J.-Y. and Lu, P.},
  journal={Theoret. Comput. Sci.},
  volume={410},
  number={18},
  pages={1618--1628},
  year={2009},
  publisher={Elsevier}
}

@article{cai2010symmetric,
  title={On symmetric signatures in holographic algorithms},
  author={Cai, J.-Y. and Lu, P.},
  journal={Theory Comput. Syst.},
  volume={46},
  number={3},
  pages={398--415},
  year={2010},
  publisher={Springer}
}

@article{cai2008basis,
  title={Basis collapse in holographic algorithms},
  author={Cai, J.-Y. and Lu, P.},
  journal={Comput. Complexity},
  volume={17},
  number={2},
  pages={254--281},
  year={2008},
  publisher={Birkhauser Verlag Basel, Switzerland, Switzerland}
}

@article{cai2011holographic,
  title={Holographic algorithms: From art to science},
  author={Cai, J.-Y. and Lu, P.},
  journal={J. Comput. System Sci.},
  volume={77},
  number={1},
  pages={41--61},
  year={2011},
  publisher={Elsevier}
}

@article{cai2008holographic2,
  title={Holographic algorithms by Fibonacci gates and holographic reductions for hardness},
  author={Cai, J.-Y. and Lu, P. and Xia, M.},
  journal={49th Annual IEEE Symposium on Foundations of Computer Science (FOCS)},
  pages={644--653},
  year={2008}
}

@article{chertkov2006loop1,
  title={Loop calculus in statistical physics and information science},
  author={Chertkov, M. and Chernyak, V.Y.},
  journal={Phys. Rev. E},
  volume={73},
  number={6},
  pages={065102},
  year={2006},
  publisher={APS}
}

@article{chertkov2006loop2,
  title={Loop series for discrete statistical models on graphs},
  author={Chertkov, M. and Chernyak, V.Y.},
  journal={J. Stat. Mech. Theory Exp.},
  volume={2006},
  number={06},
  pages={P06009},
  year={2006},
  publisher={IOP Publishing}
}

@book{lovasz2007combinatorial,
  title={Combinatorial problems and exercises},
  author={Lov{\'a}sz, L.},
  volume={361},
  year={2007},
  publisher={American Mathematical Soc.}
}

@article{csikvari2022short,
  title={A short survey on stable polynomials, orientations and matchings},
  author={Csikv{\'a}ri, P. and Schweitzer, \'A},
  journal={Acta Mathematica Hungarica},
  volume={166},
  number={1},
  pages={1--16},
  year={2022},
  publisher={Springer}
}
\end{document}